    \documentclass[12pt,a4paper]{article}
    \usepackage[cp1251]{inputenc}
    \usepackage[english]{babel}
    \usepackage{amsmath,amssymb,amsthm}
     \usepackage[dvipdfm]{graphicx}
     \usepackage{tikz}
  \DeclareMathOperator\Inv{Inv}

\renewcommand{\ge}{\geqslant}

\theoremstyle{definition}
\theoremstyle{plain}
\newtheorem{thm}{Theorem}
\newtheorem{lem}[thm]{Lemma}

\begin{document}

\title{The Size of Generating Sets of Powers}
\author{Dmitriy Zhuk}
\date{}

\maketitle









\begin{abstract}
In the paper we prove for every finite algebra $\mathbb A$ that 
either it has the polynomially generated powers (PGP) property, or it has the exponentially generated powers (EGP) property.
For idempotent algebras we give a simple criteria for the algebra to satisfy EGP property.
\end{abstract}

\section{Main Results}

We say that an algebra $\mathbb A$ has the polynomially generated powers (PGP) property if its
$n$-th power $\mathbb A^{n}$ has a polynomial-size generating set.
That is, there exists a polynomial $p$ such that for every $n$ 
the $n$-th power $\mathbb A^{n}$ can be generated by at most $p(n)$ tuples.
We say that an algebra $\mathbb A$ has the exponentially generated powers (EGP) property if 
there exists $b>1$ and $C>0$ such that 
for every $n$ the $n$-th power $\mathbb A^{n}$ cannot be generated by less than 
$C b^{n}$ tuples.

\begin{thm}\label{firstandmaintheorem}

Suppose $\mathbb A$ is a finite algebra.
Then either $\mathbb A$ has PGP property, or 
$\mathbb A$ has EGP property.

\end{thm}

Suppose $\alpha,\beta\subsetneq A$, $\alpha\cup \beta = A$. 
We say that an operation $f\colon A^{n}\to A$ is $\alpha\beta$-projective if
there exists $j\in\{1,2,\ldots,n\}$ such that 
for every $(a_1,\ldots,a_n)\in A^{n}$ and $S\in\{\alpha,\beta\}$ we have 
$f(a_1,\ldots,a_{j-1},S,a_{j+1},\ldots,a_n)\subseteq S$.

\begin{thm}\label{IdempotentCriteria}
Suppose $\mathbb A$ is a finite idempotent algebra.
Then $\mathbb A$ has EGP property
if and only if there exist $\alpha$ and $\beta$ such that 
every operation of the algebra is $\alpha\beta$-projective.
\end{thm}

For an algebra $\mathbb A$ and a subset $X\subset A^{n}$,
by $\langle X \rangle_{\mathbb A}$ we denote the subalgebra generated by $X$.
For an integer $m$ put 
$D_{A,m} = 
\{(a_1,\ldots,a_{2m}) \mid
(a_1=a_2)\vee (a_3=a_4)\vee\dots\vee (a_{2m-1} = a_{2m})\}.$

\begin{thm}\label{trivialCriteria}
Suppose $\mathbb A$ is a finite algebra.
Then $\mathbb A$ has EGP property
if and only if 
$\langle D_{A,m}\rangle_{\mathbb A}\neq A^{2m}$
for every $m\ge |A|$.
\end{thm}

\section{Proof of Theorem 1 and Theorem 3}

For a tuple $(a_1,\ldots,a_n)$
we say that $i\in\{1,2,\ldots,n-1\}$
is a switch if $a_{i}\neq a_{i+1}$.
An algebra is called \emph{$r$-switchable}, if 
$\mathbb A^{n}$ is generated by all tuples of length $n$ with at most $r$ switches.
$\mathbb A$ is called \emph{switchable} if $\mathbb A$ is $r$-switchable for some $r$.
We assume that $A = \{0,1,\ldots,k-1\}$.
A relation $\rho\subset A^{m}$ is called \emph{nice} if 
$\rho$ is not full and for all $c_1,\ldots,c_m\in A$
$$(\exists i\colon c_{i} = c_{i+1}) \Rightarrow (c_{1},\ldots,c_{m})\in\rho.$$

By $\Inv(\mathbb A)$ we denote all invariant relations of an algebra $\mathbb A$.

It is easy to check the following lemma
\begin{lem}\label{swithabilityLemma}
Suppose a finite algebra $\mathbb A$ is switchable, then it has PGP property.
\end{lem}

\begin{lem}\label{lem1}
Suppose a finite algebra $\mathbb A$ is not switchable.
Then for every $r$ there exists a nice relation in 
$\Inv(\mathbb A)$ of arity $m\ge r$.
\end{lem}

\begin{proof}
Since $\mathbb A$ is not switchable, for every $r$ there exists $n$ such that 
$\mathbb A^{n}$ is not generated by all tuples with at most $r$ switches.
Let $H\subseteq A^{n}$ be the set of all tuples with at most $r$ switches.
Put $\sigma = \langle H\rangle_\mathbb A$. 
Let $\alpha$ be a tuple in $A^{n}\setminus \sigma$ with the minimal number of switches.
Suppose 
$$\alpha =(\underbrace{a_1,\ldots,a_1}_{n_1},\underbrace{a_2,\ldots,a_2}_{n_2},\ldots,\underbrace{a_{m},\ldots,a_m}_{n_m}),$$
where $a_{1}\neq a_2, a_{2}\neq a_3, \ldots,a_{m-1}\neq a_m$.
By the definition, $m-1>r$.
Put 
$$\rho(x_1,\ldots,x_{m}) = 
\sigma(\underbrace{x_1,\ldots,x_1}_{n_1},\underbrace{x_2,\ldots,x_2}_{n_2},\ldots,\underbrace{x_{m},\ldots,x_m}_{n_m}).$$
Since $\sigma$ contains all tuple with at most $m-2$ switches, $\rho$ is nice. This completes the proof.
\end{proof}

\begin{lem}\label{lem2}
Suppose $\mathbb A$ is a finite algebra, $\rho\in\Inv(\mathbb A)$ is a nice relation of arity $m$ such that $m> 2k^2 \cdot n^2$.
Then there exists a relation $\sigma\in \Inv(\mathbb A)$ of arity $2n+k$ such that 
$\sigma$ is not full 
and for all $c_{1},\ldots,c_{n},d_1,\ldots,d_n,e_1,\ldots,e_k\in A$ we have
$$(\exists i,j\in\{1,2,\ldots,n\}\colon c_{i} = d_{j}) \Rightarrow (c_1,\ldots,c_{n},d_{1},\ldots,d_{n},e_1,\ldots,e_k)\in\sigma.$$
\end{lem}
\begin{proof}
Since $\rho$ is not full, 
there exists a tuple $(a_1,\ldots,a_m)\in A^{m}\setminus \rho.$
We consider the sequence of pairs 
$(a_1,a_2),(a_3,a_4),(a_5,a_6),\ldots$, 
where the last pair is $(a_{m-1},a_m)$ if $m$ is even and 
$(a_{m-2},a_{m-1})$ if $m$ is odd.
We choose the most popular pair in the sequence. 
Suppose this pair is $(a,b)$ and it appears $l \ge n^2$ times.
Then we identify variables in the relation $\rho$ as follows
$$\delta(x_1,y_1,x_2,y_2,\ldots,x_l,y_l,z_{0},\ldots,z_{k-1}) =
\rho(t_{1},\ldots,t_{m}),$$
where 
$t_{i} = x_{j}$ and 
$t_{i+1} = y_{j}$ if $(a_{i},a_{i+1})$ is the $j$-th pair in the sequence that is equal to $(a,b)$;
and $t_{i} = z_{a_{i}}$ otherwise.
We can easily see that 
$(a,b,a,b,a,b,\ldots,a,b,0,1,\ldots,k-1)\notin\delta$.
It remains to define a relation $\sigma$.
This time we identify variables as follows
$$\sigma(x_1,\ldots,x_{n},y_1,\ldots,y_n,z_{0},\ldots,z_{k-1}) 
= \delta(r_{1},s_1,r_2,s_2,\ldots,r_l,s_l,z_{0},\ldots,z_{k-1}),$$
where $r_{i}\in\{x_1,\ldots,x_{n}\}$ for every $i$,
$s_{i}\in\{y_1,\ldots,y_{n}\}$ for every $i$,
and the pair $(x_{i},y_{j})$ appears at least once among the pairs 
$(r_1,s_1),(r_2,s_2),\ldots,(r_l,s_l)$ for every $i$ and $j$.
Since $l\ge n^2$, we always can do this. 
We know that 
$(\underbrace{a,a,\ldots,a}_n,\underbrace{b,b,\ldots,b}_n,0,1,2,\ldots,k-1)\notin\sigma$, hence $\sigma$ satisfies the condition of the lemma.
\end{proof}

\begin{thm}\label{MainThm}
Suppose $\mathbb A$ is a finite algebra. Then
either $\mathbb A$ is switchable, or
$\mathbb A$ has EGP property.
\end{thm}

\begin{proof}

Suppose $\mathbb A$ is not switchable. Then by Lemma~\ref{lem1} for every $r$ 
there exists a nice relation of arity $m'\ge r$ in $\Inv(\mathbb A)$.
Then by Lemma~\ref{lem2} 
for every $n$ there exists a relation 
$\sigma\in \Inv(\mathbb A)$ of arity $2n+k$ such that
$\sigma$ is not full
and for all $c_{1},\ldots,c_{n},d_1,\ldots,d_n,e_1,\ldots,e_k\in A$
$$(\exists i,j\in\{1,2,\ldots,n\}\colon c_{i} = d_{j}) \Rightarrow (c_1,\ldots,c_{n},d_{1},\ldots,d_{n},e_1,\ldots,e_k)\in\sigma.$$
Let us consider 
all relations that can be obtained from $\sigma$ by a permutation of the first $2n$ variables.
The family of all such relations we denote by $\Sigma$.
Obviously, it contains exactly $2n!$ relations.
Assume that 
$\mathbb A^{2n+k}$ is generated by tuples 
$\alpha_1,\ldots,\alpha_s$. 
Since all relations in $\Sigma$ are not full, for every relation $\sigma'\in\Sigma$ there exists 
$i\in\{1,2,\ldots,s\}$ such that $\alpha_{i}\notin\sigma'$.

Let us count how many relations from $\Sigma$ can omit a tuple $(a_1,\ldots,a_{2n+k})$.
Suppose $$\sigma'(t_1,\ldots,t_{2n},z_0,...,z_{k-1} ) = \sigma(x_1,\ldots,x_{n},y_1,\ldots,y_n,z_{0},\ldots,z_{k-1}),$$
where  $\{t_1,\ldots,t_{2n}\} = \{x_1,\ldots,x_n,y_1,\ldots,y_n\}$, 
and $\sigma'$ omits the tuple $(a_1,\ldots,a_{2n+k})$.
Obviously, if $a_i = a_j$ then either $\{t_{i},t_{j}\}\subseteq \{x_1,\ldots,x_n\}$, or 
$\{t_{i},t_{j}\}\subseteq \{y_1,\ldots,y_n\}$.
Therefore, to define a permutation of variables it is sufficient to define a mapping 
$A\to\{x,y\}$, and the order of $x$'s and $y$'s in $t_1,\ldots,t_{2n}$.
We conclude that a tuple can be omitted by at most $n!\cdot n!\cdot 2^k$ relations.
Hence $s\ge (2n!)/((n!)^{2}\cdot 2^{k})>2^{n-k}.$
This means that for every $m$ 
we need at least $2^{[(m-k)/2]-k}\ge 2^{m/2-2k}$ tuples to generate $\mathbb A^{m}$. 
It remains to put $b = \sqrt{2}$ and $C = 2^{-2k}$ in the definition of EGP property
\end{proof}

Theorem~\ref{firstandmaintheorem} follows from Lemma~\ref{swithabilityLemma} and Theorem~\ref{MainThm}.
\begin{lem}\label{LtrivialCriteria}
Suppose $\mathbb A$ is a finite algebra, $m\ge k$, $\langle D_{A,m}\rangle_{\mathbb A}= A^{2m},$
then $\langle D_{A,n}\rangle_{\mathbb A}= A^{2n}$ for every $n> m$.
\end{lem}
\begin{proof}
Suppose $\alpha = (a_1,\ldots,a_{2n})\in A^{2n}$.
Put $\alpha' = (a_1,\ldots,a_{2m})$.
Let us show that $\alpha\in \langle D_{A,n}\rangle_{\mathbb A}$.
We know that there exist tuples $\gamma_1',\ldots,\gamma_s'\in D_{A,m}$ and a term operation $f$
such that $f(\gamma_1',\ldots,\gamma_s') = \alpha'$.

Since $n>m\ge k$, 
WLOG we can assume that 
$\{a_1,\ldots,a_{2m}\} = \{a_1,\ldots,a_{2n}\}$.
Obviously, any extension of $\gamma_{i}$ with $2n-2m$ elements is a tuple from $D_{A,n}$.
Then, we can find extensions $\gamma_1,\ldots,\gamma_{s}$ such that 
$f(\gamma_{1},\ldots,\gamma_s) = \alpha$.
Hence, $\alpha\in \langle D_{A,n}\rangle_{\mathbb A}$.
\end{proof}

Now, let us prove Theorem~\ref{trivialCriteria}.
\begin{proof}
Suppose $\mathbb A$ is not switchable. Then by Lemma~\ref{lem1} for every $r\ge 2k$
there exists a nice relation $\sigma$ of arity $m\ge r$ in $\Inv(\mathbb A)$.

We want $m$ to be even. If it is odd we do the following.
Suppose $(a_1,\ldots,a_{m})\notin \sigma$,
then some element of $A$ occurs at least twice in the sequence
$a_1,a_3,a_5,\ldots,a_{m}$. Assume that $a_{2j+1} = a_{2i+1}$, where $j<i$.
Then we identify the $2j+1$-th variable and the $2i+1$-th variable of $\sigma$
to get a relation $\sigma'$ of arity $m' = m-1$.
$$\sigma'(x_1,\ldots,x_{2i},x_{2i+2},\ldots,x_{m}) =
\sigma(x_1,\ldots,x_{2i},x_{2j+1},x_{2i+2},\ldots,x_{m}).$$
If $m$ is even, we put $\sigma' = \sigma$ and $m' = m$.
We can check that $D_{A,m'/2}\subseteq \sigma'$,
therefore $\langle D_{A,m'/2}\rangle_{\mathbb A}\neq A^{m'}$.
Using Lemma~\ref{LtrivialCriteria}, we prove that
$\langle D_{A,n}\rangle_{\mathbb A}\neq A^{2n}$
for every $n\ge k$.

Suppose $\langle D_{A,m}\rangle_{\mathbb A}\neq A^{2m}$
for every $m\ge k$.
Then for every $m\ge k$ we can define a relation $\sigma_{m} = \langle D_{A,m}\rangle_{\mathbb A}$, which has all properties of
a nice relation we use in Lemma~\ref{lem2}. Thus, arguing as in Theorem~\ref{MainThm} we can prove that
$\mathbb A$ has EGP property.
\end{proof}
\section{Criteria For the Idempotent Case}

\begin{lem}\label{findBlocker}
Suppose $\mathbb A$ is a finite algebra,
$\varnothing \neq B\subsetneq A$, 
$\langle A^{n}\setminus (A\setminus B)^{n}\rangle_{\mathbb A} \neq A^{n}$ for every $n$,
then there exists $C\subsetneq A$
such that $B\subseteq C$ and 
$A^{n}\setminus (A\setminus C)^{n}\in\Inv(\mathbb A)$ for every $n$. 
\end{lem}

\begin{proof}
Let $C\subsetneq A$ be a maximal set containing $B$ such that 
$\langle A^{n}\setminus (A\setminus C)^{n}\rangle_{\mathbb A} \neq A^{n}$ for every $n$.
Put $\sigma_n = \langle A^{n}\setminus (A\setminus C)^{n}\rangle_{\mathbb A}$, let us show that 
$\sigma_{n}= A^{n}\setminus (A\setminus C)^{n}$ for every $n$.
Assume the converse. 
Then there exists 
$(a_1,\ldots,a_{n})\in \sigma_n \cap (A\setminus C)^{n}$.
Since the algebra is idempotent, 
$(a_1,\ldots,a_{n})\times A^{s} \subseteq \sigma_{n+s}$ for every $s\ge 0$.
Let $m\in\{0,1,\ldots,n-1\}$
be the maximal number such that 
$(a_1,\ldots,a_{m})\times A^{s} \not\subseteq \sigma_{m+s}$ for every $s\ge 0$.
Then for some $s'$
we have $(a_1,\ldots,a_{m+1})\times A^{s'}\subseteq \sigma_{m+s'+1}$.
Since the algebra is idempotent, 
$(a_1,\ldots,a_{m+1})\times A^{s}\subseteq \sigma_{m+s+1}$
for every $s\ge s'$.
Put $$\delta_{s+1}(x_1,\ldots,x_{s+1})= \sigma_{m+s+1}(a_1,\ldots,a_{m},x_1,\ldots,x_{s+1}).$$
By the definition of
$m$ we know that $\delta_{s+1}$ is not a full relation.
Since $\sigma_{m+s+1}$ is symmetric, 
$A^{s+1}\setminus (A\setminus \{a_{m+1}\})^{s+1}\subseteq \delta_{s+1}$.
Put 
$C' = C\cup \{a_{m+1}\}$.
By the definition
we have 
$\delta_{s}\in\Inv(\mathbb A)$ and 
$A^{s+1}\setminus (A\setminus C')^{s+1}\subseteq \delta_{s+1}$.
Therefore, 
$\langle A^{n}\setminus (A\setminus C')^{n}\rangle_{\mathbb A} \subseteq
\delta_{n}\neq A^{n}$ for every $n\ge s'+1$.
It is easy to check that the above condition holds for $n< s'+1$.
This contradicts our assumption about the maximality of $C$.
\end{proof}

\begin{lem}\label{idempotent1step}
Suppose $\mathbb A$ is a finite idempotent algebra satisfying EGP property.
Then there exists a binary reflexive relation $\rho$ 
such that $\rho$ is not full and
the relation $\sigma_{n}$ defined by 
$$\sigma_{n}(x_1,\ldots,x_{2n}) = 
\rho(x_1,x_2)\vee\rho(x_3,x_4)\vee\dots\vee\rho(x_{2n-1},x_{2n})$$
is an invariant for $\mathbb A$ for every $n$.
\end{lem}

\begin{proof}

Let us consider an algebra $\mathbb D = \mathbb A\times\mathbb A$.
Put $B = \{(a,a)\mid a\in A\}$.
By Theorem~\ref{trivialCriteria}, 
$\langle D^{n}\setminus (D\setminus B)^{n}\rangle_{\mathbb A}\neq D^{n}$
for every $n$.
By Lemma~\ref{findBlocker}, there exists 
$C\subseteq D = A\times A$ such 
that $D^{n}\setminus (D\setminus C)^{n}\in\Inv(\mathbb D)$ for every $n$.
It remains to put $\rho = C$.
\end{proof}

\begin{lem}\label{IdempotentCriteriaL1}
Suppose $\mathbb A$ is a finite idempotent algebra satisfying EGP property.
Then there exist $\alpha,\beta\subsetneq A$
such that $\alpha\cup\beta = A$, $\rho = (\alpha\times\alpha)\cup(\beta\times\beta)$,
and
the relation $\sigma_{n}$ defined by
$$\sigma_{n}(x_1,\ldots,x_{2n}) =
\rho(x_1,x_2)\vee\rho(x_3,x_4)\vee\dots\vee\rho(x_{2n-1},x_{2n})$$
is an invariant for $\mathbb A$ for every $n$.
\end{lem}

\begin{proof}
By Lemma~\ref{idempotent1step} there exists a reflexive relation $\rho_0$ such that 
the relation $\sigma_{n,0}(x_1,\ldots,x_{2n}) = \rho_0(x_1,x_2)\vee\rho_0(x_3,x_4)\vee\dots\vee\rho_0(x_{2n-1},x_{2n})$ is 
an invariant for every $n$.

First, we want $\rho_{0}$ to be a symmetric relation.
Put $\rho_{0}'(x,y) = \rho_{0}(y,x)$.
We consider $2^{n}$ relations that can be obtained 
from $\sigma_{n,0}$ by replacing some $\rho_{0}$ by $\rho_{0}'$ in the definition of $\sigma_{n,0}$.
The intersection of these relations we denote by
$\sigma_{n,1}$.
We can easily check that $\sigma_{n,1}$ is an invariant for $\mathbb A$ and 
$\sigma_{n,1} (x_1,\ldots,x_{2n})= \rho_1(x_1,x_2)\vee\rho_1(x_3,x_4)\vee\dots\vee\rho_1(x_{2n-1},x_{2n})$,
where $\rho_{1} = \rho_{0}\cap\rho_{0}'$.

Assume that 
$\rho_{1}$ is a maximal symmetric reflexive relation such that the above relation $\sigma_{n,1}$ is an invariant for every $n$.
Let $C\subseteq A$ be the set of all elements $c$ such that 
$\{c\}\times A\subseteq \rho_{1}$.
Put $B = A\setminus C$.
Let us consider the relation $\rho_{1}' = \rho_{1}\cap (B\times B)$.
Assume that $\rho_{1}'$ is not transitive, then for some $a,b,c\in A$ 
we have $(a,b),(b,c)\in \rho_{1}'$ and $(a,c)\notin \rho_{1}'$.

Put $\rho_2 =  
(\rho_{1}(x_1,b)\vee
\rho_{1}(x_1,x_2))
\wedge
(\rho_{1}(x_2,b)\vee
\rho_{1}(x_1,x_2))$ 
and 
$$\sigma_{n,2}(x_1,\ldots,x_{2n}) = 
\rho_2(x_1,x_2)\vee\rho_2(x_3,x_4)\vee\dots\vee\rho_2(x_{2n-1},x_{2n}).
$$
It can be easily checked that 
$\sigma_{n,2}\in\Inv(\mathbb A)$.
Since $b\not\in C$, there exists $d\in B$ such that 
$(b,d)\notin \rho_1$.
Obviously, $(b,d)\notin \rho_{2}$, hence $\rho_{2}$ is not full.
Also, $\rho_{2}$ is a symmetric reflexive relation such that $\rho_{1}\subsetneq \rho_{2}$.
This contradicts our assumption about the maximality of $\rho_{1}$. 
Therefore, $\rho_{1}'$ is transitive, and therefore, it is an equivalence relation on $B$.
Assume that there are at least 3 different equivalence classes.
We choose some element $b\in B$ and then define $\rho_2$ as above.
We can easily see that $\rho_{2}$ is not full, which contradicts the maximality of $\rho_1$.
Hence there are exactly 2 equivalence classes. This completes the proof.
\end{proof}

\begin{lem}\label{IdempotentCriteriaL2}
Suppose $\alpha,\beta\subsetneq A$, $\alpha\cup\beta = A$, $\rho = (\alpha \times \alpha) \cup (\beta\times\beta)$,
the relation $\sigma_{n}$ is defined by
$$\sigma_{n}(x_1,\ldots,x_{2n}) =
\rho(x_1,x_2)\vee\rho(x_3,x_4)\vee\dots\vee\rho(x_{2n-1},x_{2n}).$$
Then an idempotent operation $f$ is $\alpha\beta$-projective if and only if
it preserves $\sigma_{n}$ for every $n$.
\end{lem}
\begin{proof}
Suppose $f$ is $\alpha\beta$-projective.
Let $\gamma_1,\ldots,\gamma_{s}\in \sigma_n$.
We need to prove that $f(\gamma_1,\ldots,\gamma_s)=\delta\in \sigma_n$.
Choose $j$ in the definition of $\alpha\beta$-projectiveness.
Since $\gamma_j\in\sigma_n$ we have 
$(\gamma_{j}(2r-1),\gamma_{j}(2r))\in\rho$ for some $r$.
Hence $\gamma_{j}(2r-1),\gamma_{j}(2r)\in \alpha$ or
$\gamma_{j}(2r-1),\gamma_{j}(2r)\in \beta$.
By the definition of $\alpha\beta$-projectiveness, 
we obtain 
$\delta(2r-1),\delta(2r)\in \alpha$ or
$\delta(2r-1),\delta(2r)\in \beta$. Hence $\delta\in\sigma_{n}$.

Suppose $f$ of arity $s$ preserves $\sigma_{n}$ for every $n$.
Assume that $f$ is not $\alpha\beta$-projective.
Hence for every $j\in\{1,2,\ldots,s\}$
there exists a tuple 
$(a_1^{(j)},\ldots,a_s^{(j)})$ and $S_{j}\in\{\alpha,\beta\}$ such that 
$a_{j}^{(j)}\in S_{j}$ and 
$f(a_1^{(j)},\ldots,a_n^{(j)})\notin S_{j}$.
Choose $S_{j}'$ such that $\{S_{j},S_{j}'\} = \{\alpha,\beta\}$,
and $c_{j}\in S_{j}\setminus S_{j}'$.
The following statement gives a contradiction and completes the proof
$$
f\left(\begin{matrix}
a_1^{(1)}& a_2^{(1)}& \dots & a_s^{(1)}\\
c_1 & c_1 & \dots & c_1\\
a_1^{(2)}& a_2^{(2)}& \dots & a_s^{(2)}\\
c_2 & c_2 & \dots & c_2\\
\vdots & \vdots & \ddots & \vdots\\
a_1^{(s)}& a_2^{(s)}& \dots & a_s^{(s)}\\
c_s & c_s & \dots & c_s
\end{matrix}\right)\notin \sigma_{s}.
$$
\end{proof}

Theorem~\ref{IdempotentCriteria} follows from 
Theorem~\ref{trivialCriteria}, Lemma~\ref{IdempotentCriteriaL1}, and Lemma~\ref{IdempotentCriteriaL2}.

%




\end{document}